\documentclass[a4paper, 11pt]{article}

\usepackage{authblk}

\title{
	Boundaries of Hypertrees, and Hamiltonian Cycles\\ 
	in Simplicial Complexes}

\author[1]{Rogers~Mathew}
\author[2]{Ilan~Newman}
\author[2]{Yuri~Rabinovich}
\author[2]{Deepak~Rajendraprasad}

\affil[1]{
 	Department of Computer Science, Indian Institute of Technology Kharagpur, India.
}
\affil[2]{
 	Department of Computer Science, University of Haifa, Israel.
}

\usepackage[margin=2cm]{geometry}
\newif\ifRecompileFigures
\RecompileFiguresfalse
\ifRecompileFigures
	\usepackage{pstricks}
	\usepackage{pstricks-add}
	\psset{gridcolor=yellow, gridlabelcolor=yellow, subgridcolor=yellow,%
\else
	\usepackage{graphicx}
\fi

\usepackage{amsmath,amsthm, amssymb}

\newtheorem{theorem}{Theorem}[section]

\newtheorem{observation}[theorem]{Observation}
\newtheorem{corollary}[theorem]{Corollary}

\newtheorem{lemma}[theorem]{Lemma}
\newtheorem{claim}[theorem]{Claim}

\theoremstyle{definition}
\newtheorem{definition}[theorem]{Definition}

\theoremstyle{remark}
\newtheorem*{remark}{Remark}

\newtheoremstyle{insideclaims}
  {5pt}
  {5pt}
  {\itshape}
  {}
  {}
  {}
  {0.5em}
  {{\itshape \thmname{#1}} {\itshape \thmnumber{#2}}{\itshape\thmnote{#3}}.}
\theoremstyle{insideclaims}
\newtheorem{insideclaim}{Claim}[theorem]

\def\into{\longrightarrow}
\def\N{\mathbb{N}}
\def\Z{\mathbb{Z}}
\def\Q{\mathbb{Q}}

\def\F{\mathbb{F}}

\newcommand{\size}[1]{\vert #1 \vert}
\newcommand{\ceil}[1]{\left\lceil #1 \right\rceil}

\DeclareMathOperator{\Link}{Link}
\DeclareMathOperator{\Cone}{Cone}
\DeclareMathOperator{\Ext}{Ext}
\DeclareMathOperator{\cora}{co-rank}
\newcommand{\cupdot}{\oplus}
\newcommand{\bigcupdot}{\bigoplus}
\DeclareMathOperator{\Cut}{Cut}
\DeclareMathOperator{\Fill}{Fill}

\DeclareMathOperator{\supp}{supp}

\def\boundary{\partial}

\def\Xan{X}
\def\Yan{Y}
\def\C{\mathcal{C}}

\begin{document}

\maketitle

\begin{abstract}
A {\em $d$-hypertree} on $[n]$ is a maximal acyclic $d$-dimensional simplicial complex with full $(d-1)$-skeleton on the vertex set $[n]$. 
Alternatively, in the language of algebraic topology, it is a minimal $d$-dimensional simplicial complex $T$ (assuming full $(d-1)$-skeleton) such that $\tilde{H}_{d-1}(T;\F)=0$. 

The $d$-hypertrees are a basic object in combinatorial theory of simplicial complexes.
They have been studied; and yet, many of their structural aspects remain poorly understood.

In this paper we study the boundaries $\partial_d T$ of $d$-hypertrees, and the fundamental $d$-cycles defined by them. 
Our findings include: 

\begin{itemize}
\item 
A full characterization of $\partial_d T$ over $\F_2$ for $d \leq 2$, and some partial results for $d \geq 3$.

\item
Lower bounds on the maximum size of a largest simple $d$-cycle on $[n]$. 
In particular, for $d=2$, we construct a {\em Hamiltonian $d$-cycle} $H$ on $[n]$, 
i.e., 
a simple $d$-cycle of size ${{n-1} \choose d} + 1$. 
For $d\geq 3$, we construct a simple $d$-cycle of size ${{n-1} \choose d} - O(n^{d-2})$.

\item 
Observing that the maximum of the expected distance between two vertices chosen uniformly at random in a tree ($1$-hypertree) on $[n]$ is at most $\thicksim n/3$,
attained on Hamiltonian paths, we ask a similar question about $d$-hypertrees. 
{\em How large can be the {\em average} size of a fundamental cycle of a $d$-hypertree $T$ 
(i.e., the expected size of the dependency created by adding a $d$-simplex on $[n]$, chosen uniformly at random, to $T$)}? 
For every $d \in \N$, we construct an infinite family of $d$-hypertrees $\{T\}$ with the average size of a fundamental cycle 
at least $c_d\, |T| \,=\, c_d\,{n-1 \choose d}$, where $c_d$ is a constant depending on the dimension $d$ alone.    
\end{itemize}

\end{abstract}

\section{Introduction}

Graph Theory, and in particular its portion dealing with connectivity-related notions such as 
cycles, trees and cuts, and their packing, covering, polytopes, etc., plays a fundamental 
role in CS, and is at the core of Combinatorial Optimization and the theory of Algorithms. 
While graphs are well-suited for modeling systems with pairwise interactions, 
modeling of more complex multiway interactions is required in fields like
Game Theory, Distributed Systems, Image Processing, etc. 
This calls for a robust higher-dimensional generalization 
of graphs. Hypergraphs provide a partial answer, however, it will be noted 
that the graph-theoretic connectivity notions do not even have a widely-agreed formulation 
in this context, much less a coherent theory.
   
In contrast, simplicial complexes posses a richer structure than
hypergraphs, and do allow a natural and meaningful generalization 
of trees, cycles, etc. In fact, both simplicial {\em homology} and {\em homotopy} 
theories can be viewed as high-dimensional connectivity theories. This paper is
dedicated to the study of the combinatorial structure of 
these generalizations, namely $d$-hypertrees and $d$-cycles.

\paragraph{On simplicial complexes and their applications.}
Simplicial complexes have been introduced by Poincar\'{e} at the turn of the 20'th
century as a tool for creating a rigorous foundation for what has later evolved into the {homology} theory for smooth manifolds.  The introduction of clear combinatorial notions had immediately led to flourishing of Combinatorial Topology. Although later the simplicial 
methods were gradually replaced by a powerful algebraic apparatus,  simplicial complexes have retained their importance in this branch of mathematics up to this day.
Gradually, simplicial complexes along with the associated topological methods spread through 
virtually all branches of Discrete Math.~and Computer Science%
\footnote{%
See, e.g., the excellent book of Matou$\rm\check{s}$ek~\cite{matousek2008borsuk}, and also~\cite{bjorner1995topological}, for some beautiful examples of uses of topological methods in 
Graph Theory and vice versa. The most famous applications of simplicial complexes in Theoretical CS include the study of evasiveness of nontrivial monotone graph properties of Kahn et al.~\cite{kahn1984topological}, and the consensus problem in the area of asynchronous concurrent computation~\cite{herlihy1999topological,saks2000wait}
}. 

In recent years, the study of simplicial complexes has continued at an
ever accelerating rate. A series of brilliant papers dedicated to the study
of hard topological invariants like embedding dimension and homotopy groups, see e.g.,~\cite{matousek2011hardness,cadek2013extending,matousek2014embeddability},
has significantly promoted our understanding in this direction.
Spectacular advances in the study of threshold phenomena 
in random complexes, see e.g.~\cite{linial2006homological,babson2011fundamental,aronshtam2013collapsibility,aronshtam2015TopHomology,aronshtam2015threshold,linial2014phase}, have virtually shaped a new important
topological-combinatorial area out of thin air. 
Another important similar development had to do with constructing expanders of various types,
see e.g.,~\cite{gromov2010singularities,fox2011overlap,kaufman2014ramanujan,lubotzky2015random}.
In practical CS, topological methods for data analysis have become
exceedingly popular. For an overview of this field see, e.g., the 
summaries~\cite{carlsson2009topology,ghrist2014elementary}. In addition to classical methods
of reconstructing geometrical data from properly produced samples (see,
e.g.,~\cite{dey2006curve}), new methods based on {\em persistent homology} have emerged, and become widespread. 

To sum up, the theory of simplicial complexes provides valuable tools
for a large variety of theoretical and practical areas. Situated at the meeting place of Geometry and Combinatorics, it provides a language  understandable to both, and  serves as an important vehicle for exchange of ideas and results between the two areas.

\paragraph{Our contribution.} 
In view of the above, it is surprising that the 
combinatorial structure of simplicial complexes is rather poorly understood at the moment, 
and even the simplest questions about the structure of the basic objects such 
as $2$-hypertrees, $2$-cycles and $2$-cocycles, run into unknown territory. 

A word about $d$-hypertrees is due. While $d$-cycles and $d$-cocycles are standard and very basic
notions of Combinatorial/Algebraic Topology requiring no introduction, the notions of hypertree, although standard, is less common. A $d$-hypertree $T$ is a maximal $d$-acyclic subcomplex of the complete
$d$-dimensional complex $K_n^d$ on the vertex set $[n]$. This definition implies, in particular,
that $T$ has a full $(d-1)$-skeleton. Equivalently, a $d$-hypertree is a minimal subcomplex 
$K_n^{d-1} \subset T \subseteq K_n^d$ such that $\tilde{H}_{d-1}(T,\F)=0$ with respect to the
underlying field $\F$. Hypertrees were introduced and first studied by Kalai in~\cite{kalai1983enumeration}, and were 
subsequently studied, e.g., in~\cite{linial2006homological,meshulam2009homological,linial2010SumComplexes,newman2013multiplicative,newman2015connectivity}. 

This paper (and its companion~\cite{linial2014Shadows}) is close in spirit to the 
afore-mentioned research of the threshold probabilities of various properties of random $d$-complexes. We study a number of extremal-type problems pertaining to the combinatorial/topological structure of $d$-hypertrees and $d$-cycles.   

Our first goal is the characterization of the boundaries of $d$-hypertrees.
Although at first glance it may sound a rather exotic question, its solution provides an 
approach to solving mainstream problems about $d$-hypertrees and
$d$-cycles. For $d=2$ we obtain a complete characterization;
for larger $d$'s an approximate answer is provided.

Next, we study simple $d$-cycles of the maximum size over the vertex set $[n]$.
Clearly, the size of such a cycle cannot exceed the size of a $d$-hypertree
by more than 1. Call such $d$-cycles {\em Hamiltonian}. As in graphs, and for the same reason, all the $d$-hypertrees
over $[n]$ have the same size ${n-1 \choose d}$.
Call also a $d$-hypertree Hamiltonian
if it is obtainable from a Hamiltonian $d$-cycle by removing one of its $d$-simplices.   
Do such $d$-cycles and $d$-hypertrees exist?  We show that for $d=2$ the answer  
is positive if and only if $n\equiv 0$ or $3 \mod 4$. For larger $d$'s we can only 
show that there exist simple $d$-cycles of size at least ${n-1 \choose d} - O(n^{d-2})$.

Hamiltonian $d$-hypertrees as defined above are just one possible generalization of
Hamiltonian paths in graphs. For another, metric generalization, observe that
Hamiltonians paths are the extremal (connected) graphs with respect to the expected distance between a random pair of vertices $u,v$, attaining value $\thicksim n/3$. Interpreting the distance between the vertices $u$ and $v$ in a tree $T$ as the size (minus 1) of the fundamental cycle obtained by adding the edge $(v,u)$ to it, one naturally arrives to the following problem about $d$-hypertrees. How large can be the average size of an fundamental cycle of a $d$-tree $T$ (i.e, the dependency created by adding a random uniform $d$-simplex on $[n]$ to $T$)?  
For every $d\in \N$, we explicitly construct a $d$-hypertree $T$ 
attaining value $\geq c_d\,{{n-1} \choose {d}}$ for some $c_d$ depending only on $d$.
I.e., a constant fraction of the fundamental cycles with respect to $T$ are of size commensurate 
with $T$. 

We work mostly with an interesting subclass of collapsible $d$-hypertrees, systematically employing a technique which we call the conical extension.
Both this class of $d$-hypertrees, and the technique used, are potentially useful
for other hypertree-related constructions.

\subsection{Simplicial complexes}
\label{sectionSimplicialComplexes}

\paragraph{Some standard notation.} We use $\F_2$ to denote the two element finite field and $\Q$ to denote the field of rational numbers. The notation $[n]$ is a shorthand for the set $\{1, \ldots, n\}$. If $A$ and $B$ are sets, then $A \oplus B$ denotes their symmetric difference and if $A$ and $B$ are vectors over $\F_2$, then it denotes their vector sum. 

A {\em $d$-dimensional simplex} ({\em $d$-simplex} or {\em $d$-face} for short) is a set with $d+1$ elements.  A {\em simplicial complex} $X$ is a collection of simplices that is closed under containment. That is, if $A \in X$ then every subset of $A$ also belongs to $X$. The union of all the simplices in $X$ is called the {\em vertex-set} $V(X)$ of $X$. In this article, we will always assume that $V(X)$ is finite and identify it with $[n]$, where $n = |V(X)|$. The {\em dimension} of a simplicial complex $X$ is the largest dimension among all the simplices in $X$. Further, $X$ is called {\em pure} if all its maximal faces are of the same dimension.
The set of $i$-dimensional simplices of $X$ is denoted by $X^{(i)}$. The {\em complete $n$-vertex $d$-dimensional simplicial complex} $K^d_n = \{ \sigma \subset [n] : |\sigma| \leq d+1\}$ contains all the simplices in $[n]$ with dimension at most $d$. 

\paragraph{Chains.} 
Given a field $\F$ and a simplicial complex $X$, an $\F$-weighted formal sum $Z$ of the $d$-faces in $X$ is called a {\em $d$-chain} over $\F$, i.e., $Z = \sum_{\sigma \in X^{(d)}}{c_{\sigma} \sigma}$, where $c_{\sigma} \in \F, \forall \sigma \in X^{(d)}$. The set $\supp(Z)$ of $d$-simplices with a non-zero weight in $Z$ is called its {\em support} and $|\supp(Z)|$ is called the {\em size} $|Z|$ of $Z$. 
The vertex-set of $X$ is also referred to as the vertex-set of $Z$. The collection of all $d$-chains of $K^d_n$ form a vector space $\C_d$ over $\F$ with dimension ${n \choose d+1}$. 

\paragraph{The boundary operator.}
The {\em boundary} $\boundary_d \sigma$ of a $d$-simplex $\sigma = \{v_0, \ldots, v_d\}$, with $v_0 < \cdots < v_d$, is the $(d-1)$-chain $\sum_{i=0}^d (-1)^i (\sigma \setminus \{v_i\})$. The signs disappear when one works over $\F_2$. The linear extension of the same to the whole of $\C_d$ is the {\em boundary operator} $\boundary_d : \C_d \into \C_{d-1}$. It is a well known and easily verifiable fact that $\boundary_{d-1} \boundary_d = 0$. The subscript of the operator may be omitted when unambiguous. 

\paragraph{Cycles, forests and trees.}
A $d$-chain $Z$ is called a {\em $d$-cycle} if $\boundary Z = 0$. We refer to $0 \in \C_d$  as the {\em trivial $d$-cycle}. Further, if there is no nontrivial cycle supported on a proper subset of $\supp(Z)$, then $Z$ is called {\em simple}. A set $F$ of $d$-faces is called a {\em $d$-forest} (or {\em acyclic}) over a field $\F$ if $F$ does not contain the support of any nontrivial $d$-cycle over $\F$. Further, $F$ is called a {\em $d$-hypertree} or a {\em $d$-tree} on $[n]$ if it is a maximal $d$-forest in $K^d_n$ under inclusion. It is easy to see that for a $d$-forest $F$, the set of $(d-1)$-chains $\{\boundary \sigma : \sigma \in F\}$ is linearly independent in $\C_{d-1}$ and when $F$ is a $d$-tree this set becomes a basis for the cycle-space in $K^{d-1}_n$, that is the kernel of $\boundary_{d-1}$. Hence all $d$-trees on $[n]$ have size ${n-1 \choose d}$. 

\paragraph{Duals and cuts.}
The {\em dual} of a $d$-chain $Z = \sum_{\sigma \in \supp(Z)} c_{\sigma} \sigma$ over $[n]$ is the $(n-d-2)$-chain $Z^* = \sum_{\sigma \in \supp(Z)} c_{\sigma} \bar{\sigma}$, where $\bar{\sigma} = [n] \setminus \sigma$. A {\em cut} is the dual of a simple cycle. Over $\F_2$, the dual of a set of $d$-simplices $Z$ over $[n]$ is the set $Z^*$ of $(n-2-d)$-simplices $\{[n] \setminus \sigma : \sigma \in Z\}$.
 
\paragraph{Degree.}
The {\em degree} of an $i$-face $\sigma$ in a simplicial complex $X$ is the number of $(i+1)$-faces in $X$ which contains $\sigma$. The {\em maximum (resp., minimum) degree} of a $d$-dimensional complex $X$ is the largest (resp., smallest) degree among all its $(d-1)$-faces. 

\paragraph{Collapsibility.}
In a $d$-dimensional simplicial complex $X$, a $(d-1)$-face $\tau$ is called {\em exposed} if its degree is $1$, that is, it belongs to exactly one $d$-face $\sigma$ of $X$. An elementary $d$-collapse on $\tau$ consists of the removal of $\sigma$ and $\tau$ from $X$. We say that $X$ is {\em $d$-collapsible} if every $d$-face of $X$ can be removed by a sequence of elementary $d$-collapses. It is easy to see that if $X$ is $d$-collapsible, then $X^{(d)}$ is a $d$-forest over any field.

\paragraph{Hamiltonian cycles and trees.}
A $d$-cycle $Z$ on $[n]$ is called {\em Hamiltonian} if it is simple and has size ${n-1 \choose d} + 1$. Notice that this is the largest possible size of a simple $d$-cycle on $[n]$. Removing any $d$-face from the support of $Z$ results in a $d$-tree. Such $d$-trees, i.e., $d$-trees contained in the support of a Hamiltonian $d$-cycle, are called {\em Hamiltonian $d$-trees}.

\paragraph{Fillings and filling-volume.}
For any $(d-1)$-cycle $Z$ and any $d$-tree $T$, both on $[n]$, there exists a unique $d$-chain $F$ supported on $T$ such that $\boundary F = Z$. This $d$-chain $F$ is called the {\em filling} of $Z$ in $T$ and is denoted by $\Fill(Z, T)$. When $Z$ above is the boundary $\boundary \sigma$ of some $d$-simplex $\sigma \in K^d_n$, we call $\Fill(\boundary\sigma, T)$ as the filling of $\sigma$ in $T$ and denote it by just $\Fill(\sigma, T)$. 
The {\em average filling-volume} $\mu(T)$ of $T$ is ${n \choose d+1}^{-1} \sum_{\sigma \in K^d_n} {|\Fill(\sigma, T)|}$.

\paragraph{Abuses of notation (\underline{Important}).}
The first type of abuse is to blur the distinction between a pure $d$-dimensional complex $X$ and its set of $d$-faces $X^{(d)}$. For example, we might say that $X$ is a forest (resp., a tree) to mean that $X^{(d)}$ is a $d$-forest (resp., a $d$-tree). In the other direction, if $Y$ is a collection of $d$-faces, by collapsibility (resp., maximum degree, minimum degree) of $Y$ we mean the collapsibility (resp., maximum degree, minimum degree) of the simplicial complex obtained by taking the subset-closure of $Y$. The second type of abuse is to blur the  distinction between a $d$-chain and its support when working over $\F_2$. We will say more about this in Section \ref{sectionPreliminaries}.

\subsection{Preliminaries}
\label{sectionPreliminaries}

\subsubsection*{Working over the field $\F_2$}

A $d$-chain over $\F_2$ can be identified bijectively with its support. Addition of two chains will correspond to the symmetric difference of their supports. Boundary of a collection of $d$-simplices $X$ is understood as the boundary of the unique $d$-chain over $\F_2$ whose support is $X$. It may be helpful to note that, in this case, the boundary of a $d$-simplex is just the collection of $(d-1)$-simplices contained in it and thus $\boundary X$ is the collection of $(d-1)$-faces with an odd degree in $X$. 
%
Henceforth in this article, we will work only over $\F_2$, although our results extend to $\Q$ (cf. Section~\ref{sectionRationalRemarks}).

\subsubsection*{Basic operators and related notions}


Besides the boundary operator $\boundary$ discussed above, additional standard
and less standard operators will be used in this paper.

\paragraph{Link.} The {\em link} operator maps a simplicial complex $X$ to the neighborhood 
of a vertex $v$:
\[
\Link_v (X) ~=~ \{ \sigma \setminus \{v\} ~|~ \sigma \in X, ~v \in \sigma\}\,.
\]
With a slight abuse of notation, we shall treat $\Link_v(X^{(d)})$ as an operator 
mapping a collection $X^{(d)}$ of $d$-simplices over $[n]$, to a collection
of $(d-1)$-simplices over $[n]\setminus \{v\}$. Observe that a link of 
a $d$-cycle $Z$ is a $(d-1)$-cycle, since in this case 
$\Link_v(Z) = \boundary \{\sigma \;|\; i \in \sigma \in Z\}$, 
and $\boundary\boundary = 0$.

\paragraph{Cone.} The {\em cone} operator is the right inverse of the link operator;
it maps a collection $X^{(d)}$ of $d$-simplices over $[n]$, to a collection
of $(d+1)$-simplices over $[n] \cup \{v\}$, where $v \notin [n]$: 
\[
\Cone_v(X^{(d)}) ~=~ \left\{\,\sigma \cup \{v\} ~|~ \sigma \in X^{(d)}\, \right\}\;.
\]
It is easy to verify that the boundary of $\Cone_v(Y)$ is given by
\begin{equation}
	\boundary \Cone_v(Y)  ~=~ Y \;\oplus\; \Cone_v (\boundary Y). 
\label{eqnBoundaryLift}
\end{equation}

\paragraph{Conical extension.} Let $X^{(d)}$ and $Y^{(d-1)}$ be collections of $d$- and $(d-1)$-simplices,
respectively, over $[n]$, and let $v \not\in [n]$. We define the {\em conical extension}
$\Ext_v(X^{(d)}, Y^{(d-1)})$, a collection of $d$-simplices on $[n] \cup \{v\}$, by
\[
\Ext_v(X^{(d)}, Y^{(d-1)}) ~=~ X^{(d)} \;\oplus\; \Cone_v(Y^{(d-1)})~.
\]
(The two summands are in fact disjoint.)
It is easy to verify using (\ref{eqnBoundaryLift}), that if both $X^{(d)}$ and 
$Y^{(d-1)}$ are acyclic, then so is $\Ext_v(X^{(d)}, Y^{(d-1)})$. Similarly,
if both $X^{(d)}$ and $Y^{(d-1)}$ are collapsible, then so is $\Ext_v(X^{(d)}, Y^{(d-1)})$.

Moreover, assume that $X^{(d)}$ and $Y^{(d-1)}$ are, respectively, $d$- and $(d-1)$-
hypertrees on $[n]$. Keeping in mind that an acyclic collection of $k$-simplices 
on $[n]$ is a $k$-hypertree if and only if its size is ${n-1 \choose k}$, we conclude that 
$\Ext_v(X^{(d)}, Y^{(d-1)})$ is a $d$-hypertree on $[n] \cup \{v\}$, since
it is acyclic, and its size is ${n-1 \choose d} + {n-1 \choose d-1} \,=\, {n \choose d}$.

This leads to the following purely combinatorial definition-theorem, fundamental for this paper:
\begin{definition}[Nice trees]
\label{def:nice}
Let $\{T_k\}_{k={d+1}}^{n-1}$ be a sequence of $(d-1)$-hypertrees, where $T_k$ has an
underlying vertex set $[k]$, respectively. Let $H_{d+1}= \{1,\ldots,d+1\}$ (a single $d$-simplex), and define 
recursively, for $k=d+1,\ldots, n-1$,
\[
H_{k+1} \;=\; \Ext_{k+1}(H_k, \,T_k)\,. 
\]
Then, each $H_k$ is a $d$-hypertree on $[k]$. Call such hypertrees {\em nice}.\\
They form a subfamily of collapsible $d$-hypertrees. For $d=1$, all trees are nice.
\end{definition}

\paragraph{Filling and conical extension.}
The last operator to be discussed in this section is $\Fill(*,T)$, a bijection
between the $(d-1)$-cycles on $[n]$ and subsets of $d$-simplices of a fixed $d$-tree $T$.
Let $Z \in {\cal Z}_{d-1}$ be a $(d-1)$-cycle on $[n]$. Then, by the definition of $d$-hypertree, there exits a unique subset $F$ of $d$-simplices in $T$, such that 
$\boundary F = Z$. We define $\Fill(Z,T) = F$. Due to acyclicity of $T$, this is indeed a bijection. Observe that the operator $\Fill(*,T)$ is linear: 
~$\Fill(Z_1 \oplus Z_2,T) = \Fill(Z_1,T)  \oplus \Fill(Z_2,T)$.
Finally, to put things in a graph-theoretic perspective, observe that for 
a $d$-simplex $\sigma$ on $[n]$, $\Fill(\boundary \sigma,T) \oplus \sigma$ is 
an analogue of what is called in graph theory {\em a fundamental cycle of $\sigma$ 
with respect to $T$}. 

The following claim about fillings of conical extensions will be used in Section~\ref{sectionCapSize}.

\begin{claim}
\label{cl:captrick}
Let $T_{n-1}^d$ and $T_{n-1}^{d-1}$ be, respectively, a $d$-hypertree and a $(d-1)$-hypertree on $[n-1]$. They define, by means of conical extension, a new
$d$-hypertree $T_n^d$ on $[n]$:~ $T_n^d \;=\; \Ext_n (T_{n-1}^d, T_{n-1}^{d-1})$.
For any $Z \in {\cal Z}_{d-1}$, a $(d-1)$-cycle on $[n]$, it holds that
\[
\Fill(Z,T_n^d) ~=~ \Cone_n(F) ~\oplus~ \Fill(Z' ,\; T_{n-1}^{d})\,, 
\]
where ~$F \,=\, \Fill\left(\Link_n(Z), {T_{n-1}^{d-1}}\right)$ and $Z' = Z \oplus \Cone_n \Link_n(Z) \oplus F$.
\end{claim}
\begin{proof} Observe that since $\Link_n(Z)$ is a $(d-1)$-cycle on $[n-1]$,
$F$ is well defined. Observe also that since $F$ is a subset of $T_{n-1}^{d-1}$,
and $T_{n-1}^{d}$ is a subset of $T_n^d$, the right-hand-side of the definition 
is a subset of $T_n^d$. Thus, due to acyclicity of $T_n^d$, it suffices to show that 
the boundary of the right-hand-side is $Z$. 

By (\ref{eqnBoundaryLift}) and the definition of the $\Fill$ operator, 
\[
\boundary \Cone_n(F) ~=~ F \oplus \Cone_n(\boundary F) ~=~ 
F \oplus \Cone_n \Link_n(Z) \,,
\]
and 
\[
\boundary\left( \Fill(Z', T_{n-1}^{d})\right) 
~=~ Z' ~=~ Z \oplus \Cone_n \Link_n(Z)  \oplus F\,.
\]
The $\oplus$ of the two terms is indeed $Z$.
\end{proof}

\section{The boundaries of hypertrees}
\label{sectionBoundary}


The central result of this section is the following characterization of the
boundaries of $2$-hypertrees:

\begin{theorem}
\label{corollary2DGivenBoundary}
A nontrivial $1$-cycle $Z$ on $[n]$ is the boundary of a $2$-hypertree on $[n]$
if and only if ~$\size{Z} \equiv {n-1 \choose 2} \mod 2$.
\end{theorem}

Observe that the trivial cycle cannot be the boundary of an acyclic
set. The theorem will be proved by first establishing the following much more general (but less precise) result.

Given an acyclic set $F$ of $d$-simplices on $[n]$, henceforth to be called
a {\em $d$-forest}, let $\cora(F) = {n-1 \choose {d}} - |F|$. In particular, if $F$
is a $d$-hypertree, then $\cora(F) = 0$. 

We shall use the convention that ${n \choose k} = 0$ for $k<0$.

\begin{theorem}
\label{theoremLargeForestGivenBoundary}
For every $d\geq 1$, and every nontrivial $(d-1)$-cycle $Z_n^{d-1}$ over $[n]$, there exists a collapsible $d$-forest $F_n^d$ on $[n]$ such that $\boundary F_n^d =  Z_n^{d-1}$, and 
$\cora(F_n^d) \leq {n-1 \choose d-2}$, which is typically much smaller than ${n-1 \choose d}$.~
Moreover, there are at least $n-2d$ different such $F_n^d$'s.
\end{theorem}

\begin{proof}
The proof proceeds by a double induction on $d$ and $n$. 

Let us first verify the statement for $d = 1$ by an induction on $n$. It is, of
course, a simple exercise; the reason we spell it out here is that the proof of the general case will have a very similar structure.  

For $n=1$ the statement is vacuous, and for $n=2$ it is trivial. So we assume $n \geq 3$, and that the statement is true for all $k < n$.
Given a nontrivial $0$-cycle $Z^0_n \subseteq [n]$ (i.e., a nonempty set of vertices
of even size) as the required boundary, we may assume without loss of generality (by relabelling, if necessary) that  
$n \in Z^0_n$. Pick any $i \in [n-1]$ such that $Z^0_n \neq \{i,n\}$, and set $Z^0_{n-1} = Z^0_n \oplus \{i, n\}$, a nontrivial $0$-cycle on $[n-1]$.
Let $T^1_{n-1}$ be a tree on $[n-1]$ with boundary $Z^0_{n-1}$. Its existence
is ensured by the induction hypothesis. Attaching $n$ to $i$ in $T^1_{n-1}$ 
by a new edge $(i,n)$, yields the required tree $T^1_n$ with boundary $Z^0_n$. Moreover, the $(n-2)$ possible choices of $i$ yield $(n-2)$ different trees with this property. 

Having established the base case of our induction, we proceed to prove the general case
$(d,n)$, $d>1$, assuming that the statement is true for all $n$ in dimensions 
smaller than $d$, and for all $(d,k)$ with $k<n$. 

When $d+1 \leq n \leq 2d-1$, it holds that ${n-1 \choose d} \leq {n-1 \choose d-2}$, and so the co-rank of any $d$-forest on $[n]$ is at most ${n-1 \choose d-2}$. Thus, $F_n^d  = \Fill(Z_n^{d-1},T_n^d)$, where $T_n^d$ is any collapsible $d$-hypertree, is a collapsible $d$-forest satisfying the requirement $\boundary F_n^d  = Z_n^{d-1}$.
The number of such forests is obviously $\geq 1 > n - 2d$. Thus, the statement is true for 
such $n$'s.

We proceed to the induction step assuming $n \geq 2d$.  
As before, one may assume without loss of generality (by relabeling the vertices) that the vertex $n$ participates in $Z^{d-1}_n$. Define $Z^{d-2}_{n-1} = \Link_n(Z^{d-1}_n)$; it is a nontrivial $(d-2)$-cycle on $[n-1]$. Now, consider collapsible $(d-1)$-forests $Y$ on $[n-1]$ satisfying the three conditions below.

\begin{equation} \label{equationdDcondition}
	\begin{array}{rrcl}
		\textnormal{(a)} & \boundary(Y) &=& ~Z^{d-2}_{n-1}; \\
		\textnormal{(b)} & Y ~~~~&\neq& ~Z^{d-1}_n \,\oplus\, \Cone_n(Z^{d-2}_{n-1});\\
		\textnormal{(c)} & \size{Y}~~~ &\geq& ~{n-2 \choose d-1} - {n-2 \choose d-3}\,.
	\end{array}
\end{equation}

By induction hypothesis on smaller dimensions, there are at least $(n-2d+1)$ different 
$Y$'s satisfying (\ref{equationdDcondition}a) and (\ref{equationdDcondition}c). Since at most one forest can violate (\ref{equationdDcondition}b), there are at least $n-2d$ different $Y$ that satisfy all the three conditions above. 

If such $Y$'s exist, set $F_{n-1}^{d-1} = Y$ for some possible choice of $Y$. Each
such choice will result in a different final construction. Define
\begin{equation} \label{equationdDrecursion}
Z^{d-1}_{n-1} = Z^{d-1}_n \oplus \Cone_n(Z^{d-2}_{n-1}) \oplus F^{d-1}_{n-1}\,.
\end{equation} 
Observe that (\ref{equationdDcondition}a) ensures that $Z^{d-1}_{n-1}$ is a $(d-1)$-cycle on
$[n-1]$, and (\ref{equationdDcondition}b) ensures that it is nontrivial. The induction hypothesis about the case $(d,n-1)$ yields the existence of a collapsible forest $F_{n-1}^{d}$ such that $\boundary F^d_{n-1} = Z^{d-1}_{n-1}$ and $\size{F^d_{n-1}} \geq {n-2 \choose d} - {n-2 \choose d-2}$. 

If $Y$ as above does not exist, which may happen only when $n=2d$, there still exists $Y'$
satisfying (\ref{equationdDcondition}a) and (\ref{equationdDcondition}c), and in addition
$Y' = Z^{d-1}_n \,\oplus\, \Cone_n(Z^{d-2}_{n-1})$. Set $F_{n-1}^{d-1} =  Y'$.
The cycle $Z^{d-1}_{n-1}$ is still defined by (\ref{equationdDrecursion}), but now
it is trivial.  Choosing $F^d_{n-1} = \emptyset$, it still holds that $\boundary F^d_{n-1} = Z^{d-1}_{n-1}$, and that $\size{F^d_{n-1}} \geq {n-2 \choose d} - {n-2 \choose d-2} = 0$.

Having defined $F^{d-1}_{n-1}$ and $F^{d}_{n-1}$ (and now it does not matter whether $Y$ exists or not), we finally define the desired $F_d^n$\;:
\begin{equation}
\label{eq:fnd}
F^d_n  ~=~ \Ext_n(F^d_{n-1}, F^{d-1}_{n-1}) ~=~ 
F^d_{n-1} \,\oplus\, \Cone_n(F^{d-1}_{n-1})~.
\end{equation}\par
Since both $F^{d-1}_{n-1}$ and $F^{d}_{n-1}$ are collapsible, so is $F^d_n$.

Next, let us verify that $F^d_n$ has the required boundary and co-rank.
\[
\begin{array}{rcll}
	\boundary(F^d_n) 
	&=& \boundary(F^d_{n-1}) 
		\oplus F^{d-1}_{n-1} 
		\oplus \Cone_n(\boundary(F^{d-1}_{n-1}))
		& ~~~~~~~(\textnormal{by (\ref{eqnBoundaryLift}) and the linearity of $\boundary$}) \\
	&=& Z^{d-1}_{n-1} 
		\oplus F^{d-1}_{n-1} 
		\oplus \Cone_n(Z^{d-2}_{n-1})
		& ~~~~~~~(\textnormal{by definition of $F_{n-1}^{d}$ and (\ref{equationdDcondition}a})) \\
	&=& Z^{d-1}_n\,.
		&~~~~~~~ (\textnormal{by (\ref{equationdDrecursion}})) \\
\end{array}
\]

Further, we have 
\[
\begin{array}{rcl}
{n-1 \choose d} - \size{F^d_n} 
	&=& {n-1 \choose d} - \left(\size{F^d_{n-1}} + \size{F^{d-1}_{n-1}} \right)
	\\
	&=& \left( {n-2 \choose d} - \size{F^d_{n-1}} \right)
	+	\left( {n-2 \choose d-1} - \size{F^{d-1}_{n-1}} \right) 
	\\
	&\leq& {n-2 \choose d-2} + {n-2 \choose d-3}
	\\
	&=&	{n-1 \choose d-2}\,.
\end{array}
\]

Finally, it is easy to verify that each of the $n-2d$ different choices available for $F^{d-1}_{n-1}$, yields a different $F^d_n$, and so there are at least 
$n-2d$ different $F^d_n$ with the required properties. 
\end{proof}

It is now an easy matter to prove Theorem~\ref{corollary2DGivenBoundary}.

\begin{proof}~{\em (of Theorem~\ref{corollary2DGivenBoundary})}~
By Theorem~\ref{theoremLargeForestGivenBoundary} applied to $d=2$,
the size of the constructed $F_n^2$ with $\boundary F_n^2 = Z_n^1$ is
either ${n-1 \choose 2}$,  or ${n-1 \choose 2}-1$. However, since the boundary of every $2$-simplex consists of 3 edges, $|F_n^2|$ and $|\boundary F_n^2|$ must
have the same parity. Thus, $|F_n^2| = {n-1 \choose 2}$, i.e., it is a 2-hypertree,
if and only if $|Z_n^1|$ has the same parity as ${n-1 \choose 2}$.
\end{proof}

We conclude this section with the following easy corollary of 
Theorem~\ref{theoremLargeForestGivenBoundary} about the density
of boundaries of $d$-hypertrees among the $d$-cycles:

\begin{corollary}
\label{th:easy}
For every $(d-1)$-cycle $Z$ on $[n]$, there exists a $(d-1)$-cycle $Z'$ on
$[n]$, such that $Z'$ is the boundary of a $d$-hypertree on $[n] $, and 
$|Z \oplus Z'| \leq (d+1) \cdot {n -1 \choose d-2}$.
\end{corollary}

\section{\boldmath{On the maximum size of a simple $d$-cycle on $[n]$}}
\label{sectionLargestSimpleCycle}

Observe that if the boundary of a $d$-forest $F_n^d$ is of the form 
$\boundary \sigma$ for some $d$-simplex $\sigma$ on $[n]$, then 
$F_n^d \oplus \sigma$ is a simple cycle. Therefore, applying Theorem~\ref{theoremLargeForestGivenBoundary} to $(d-1)$-cycles of the
form $\boundary \sigma$, we conclude that:

\begin{theorem}
\label{th:hamiltonian}$\mbox{}$
\begin{itemize}
\item Hamiltonian $2$-cycles on $[n]$ exist over $\F_2$ if and only if $n \equiv 0 \textnormal{ or } 3 \mod 4$.
\item For general $d$, the maximum size of a simple $d$-cycle on $[n]$ is at least 
${n-1 \choose d} - {n-1 \choose d-2} +1$.   
\end{itemize}
\end{theorem}

\begin{remark}
We conclude this section by mentioning that the parity condition $\size{X} \equiv \size{\boundary(X)} \mod 2$ for even dimensions is not the only factor that restricts the size of a simple cycle. The dual of a simple cycle is a cut and when $d = n-3$, the dual of a simple $d$-cycle on $[n]$ is a graphical ($1$-dimensional) cut. Since the largest cut in a graph on $n$ vertices has size at most $n^2/4$, this is the size of a largest simple $n$-vertex $(n-3)$-cycle too. It is shown in~\cite{linial2014Shadows} that the largest $2$-dimensional cut over $\F_2$ has a size ${n-1 \choose 3} - \frac{n^2}{4} - \Theta(n)$. Hence the size of a largest simple $n$-vertex $(n-4)$-cycle over $\F_2$ is also ${n-1 \choose n-4} - \frac{n^2}{4} - \Theta(n)$. 
\end{remark}

\section{Trees with large filling-volume}
\label{sectionCapSize}

In this section our aim is to  demonstrate that, in every dimension $d$, there exist trees of average filling-volume $\Omega(n^d)$.
We construct a sequence of nice $d$-trees by conical extensions along carefully chosen relabellings of a $(d-1)$-tree of average filling-volume $\Omega(n^{d-1})$. The subtlety lies precisely in the choice and the analysis of a suitable relabelling scheme%
\footnote{%
Let us mention that, for a fixed $d$, choosing the relabellings independently and uniformly at random in each step $i > d+1$, yields a sequence $\{T_i^d\}_{i=d+2}^\infty$ of random $d$-trees, where every $T_i^d$ has the desired property with positive probability. Interestingly, more can be said in this setting. Let $Z$ be {\em any} $(d-1)$-cycle in $K_{n}^{d-1}$, and let $\tilde{T}_n^d$ be a random uniform relabeling of $T_n^d$ from this sequence. Then, the expected value of $\Fill(Z,\tilde{T}_n^d)$ is $\Omega(n^{d})$. The analysis is omitted here in favor of the explicit deterministic construction presented next.}.
 
In order to illustrate the main ideas in a more concrete setting, we describe the construction of a family of $2$-trees on $[n]$ with average filling-volume $\Omega(n^2)$. In $d = 1$, a Hamiltonian path on $[n]$ has average filling volume $\Omega(n)$.
We will show, in particular, how the problem for $2$-dimensional trees is reduced to that of constructing a special sequence of Hamiltonian paths.

We use the conical extension described in Section~\ref{sectionPreliminaries} 
multiple times to construct a $2$-tree $T^2_n$ on $[n]$, given a $2$-tree $T^2_m$ on $[m]$, $m < n$, and a sequence of $1$-trees $T^1_i$ on $[i]$, $i \in \{m, \ldots, n-1\}$, as follows:
\begin{equation} 
\label{eqnNiceSequence}
	T^2_{i+1} = \Ext_{i+1} (T^2_{i}, T^1_{i}), \; i \in \{m, \ldots, n-1\}.
\end{equation}

Specifically, we will assume that $n$ is divisible by $4$, to avoid the use of floor and ceil, and  choose $m = n/2$, and $T_m^2$ to be an arbitrary tree on $[m]$.

Our aim is to construct the sequence of $1$-dimensional (graphical) trees $T_i^1$, to ensure that for a typical $2$-simplex $\sigma = \{a,b,c\}$, $\Fill^2_n = \Fill(\boundary\sigma, T_n^2)$ will be large. For that, fixing $\sigma$, and  using repeatedly Claim~\ref{cl:captrick} we get: 

\begin{equation}
\label{eqnGeneralCap}
	\Fill^2_n = \bigcupdot_{i=n}^{m+1} \Cone_i (\Fill^1_{i-1})  \cupdot \Fill^2_m,
\end{equation}
where, going by decreasing $i$,   $Z_n$ is the (graphic) triangle  = $\{(a,b),(b,c),(c,a)\} = \boundary \sigma$,   
$\Fill^1_{i-1} = \Fill(\Link_{i} (Z_i), T_{i-1}^1)$, for 
$Z_i = Z_{i+1} + \boundary \Cone_{i+1} (\Fill^1_i)$, and
$\Fill^2_m = \Fill(Z_m, T^2_m)$. 
It may be noted that the summands in Equation~\ref{eqnGeneralCap} are a disjoint collection of $d$-simplices.

Since $\Fill^2_m$ is arbitrary, to obtain a large filling-volume, we rely solely on the sum $\sum_{i=m+1}^n |\Cone_i (\Fill^1_{i-1}) | = \sum_{i=m+1}^n |\Fill^1_{i-1}|$. Hence it is enough to understand the sequence $\Fill^1_{i-1}$, $i=n , \ldots, m+1$. 

Now, $Z_i$ is a $1$-dimensional cycle (that is, a graphic cycle). In our case, by construction, we will ensure that it will always be a simple cycle. Hence $\Link_i (Z_i)$ are the two neighbors $x,y$ of $i$, in $Z_i$ if $i \in Z_i$, and $\emptyset$ if $i \notin Z_i$. In the latter case, we gain nothing to the sum above, hence we will design $Z_i$ so that $i$ will belong to $Z_i$ for many $i$'s. 

Suppose first that $i \in Z_i$, with neighbors $x,y$ in $Z_i$. Then $\Fill(\{x,y\}, T_{i-1}^1)$ is the path in $T_{i-1}^1$ from $x$ to $y$, and in order to contribute significantly to the sum, we wish that this will be large. Obviously for $T_{i-1}^1$, being a Hamiltonian path, the ``typical" filling-volume will be as large as possible, which motivates the choice of $T^1_{i-1}$ to be a Hamiltonian path on $[i-1]$.  

Let us examine a concrete construction.  Let $A = \{1, \ldots, \frac{n}{4}\}$, $B = \{\frac{n}{4}+1, \ldots, \frac{n}{2}\}$ and $C = \{\frac{n}{2}+1, \ldots, n-1\}$, then each $T^1_i$ is a path of the form $P_AiP_BP_C$, where $P_A$, $P_B$ and $P_C$ are Hamiltonian paths on $A$, $B$ and $C \cap [i-1]$, respectively (Figure~\ref{figTwoDTree}). Most importantly, note how $P_A$ and $P_B$ change with the parity of $i$.  

\begin{figure}[t]

%
\begin{center}
\ifRecompileFigures
	\input{SocG--Full-fig-2.tex}
\else
	\includegraphics[width=\textwidth]{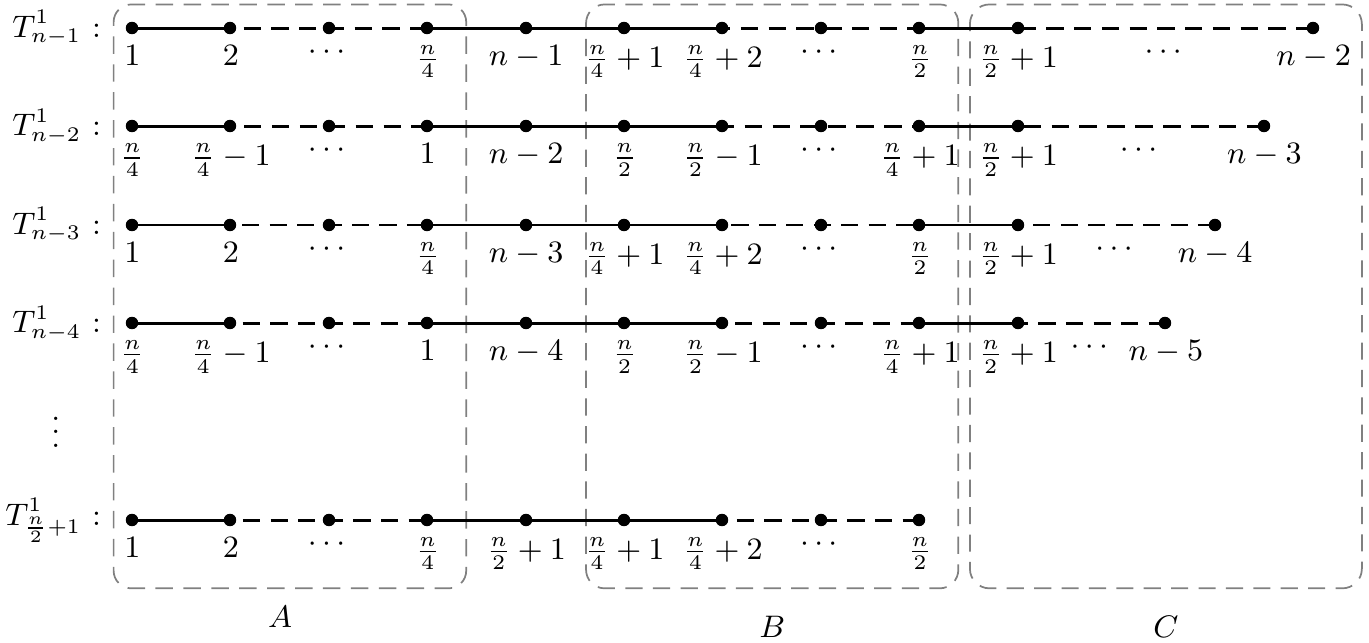}
\fi
\end{center}
\caption{The sequence of $1$-trees $\{T^1_i\}$ used to construct a $2$-tree on $n$ vertices with average filling-volume  $\Omega(n^2)$.}
\label{figTwoDTree}

\end{figure}

We call a $2$-simplex $\sigma = \{a,b,c\}$ ``good'' if $a \in A$, $b \in B$ and $c \geq n/2 + 2$. Note that for a good $\sigma = \{a,b,c\}$, if $c <n$,  $c \notin Z_n = \boundary\sigma$ and hence $Z_{n-1} = Z_n = \sigma$. Thus this carries on as long as $i > c$. For $i=c$, $\Link_i (Z_i) = \{a,b\}$, and then $Z_{i-1} = \{(a,b)\} \cup \Fill(\{a,b\}, T_i^1)$ which is, by construction, the simple cycle containing the suffix of $P_A$ from $a$ to $i-1$, then the prefix of $P_B$ from $i-1$ to $b$, plus the edge $(a,b)$. In particular, it is simple and contains $i-1$. Note also, that by the alteration between even and odd values of $i$, from $i=c$ down to $n/2+1$,  $\Link_j(Z_j)$  are $\{1,\frac{n}{2}\}$ for even $j$, and $\{\frac{n}{4}, \frac{n}{4} +1\}$ for odd $j$,  for every $\frac{n}{2} +1 < j < c$.  This is summed up in the following:

\begin{claim}
\begin{equation}
\Link_i(Z^1_i) =
	\begin{cases} 
	\begin{array}{ll}
		\emptyset, 			& i \in \{c+1, \ldots, n\}, \\
		\{a,b\},			& i = c,\\
		\{\frac{n}{4}, \frac{n}{4} + 1\}, 	& i \in \{\frac{n}{2} + 1, \ldots, c-1\}, \;i \textnormal{ odd}, 
			\textnormal{ and} \\
		\{1, \frac{n}{2}\}, 			& i \in \{\frac{n}{2} + 1, \ldots, c-1\}, \;i \textnormal{ even}.
	\end{array}
	\end{cases}
\end{equation}
Moreover,
\begin{equation}
|\Fill^1_i| = \frac{n}{2}+1, ~~ i \in \left\{\frac{n}{2}+1, \ldots, c-1 \right\}.
\end{equation}
\end{claim}

Hence, equation \ref{eqnGeneralCap} implies that for a good $\sigma$,
$|\Fill(\{a,b,c\}, T^2_n)| \gtrapprox (c - \frac{n}{2} -2)\frac{n}{2}$. 

For any $c \geq n/2 + 2$, the number of good triangles containing $c$ is $|A||B| = (\frac{n}{4})(\frac{n}{4})$. Hence 

\begin{equation}
\begin{array}{rcl}
\sum_{\sigma ~good}  |\Fill(\sigma, T^2_n)| 
	&\gtrapprox& 	\sum_{c = n/2+3}^{n} 
		\frac{n^2}{16}
		\left(c-\frac{n}{2}-2\right) 
		\frac{n}{2}\\
	&\gtrapprox& n^5/2^9.
\end{array}
\end{equation} 
Since the total number of $2$-faces is less than $n^3$, we see that $\mu(T^2_n) \in \Omega(n^2)$. 

After introducing one more notation, let us briefly discuss the essential features in the construction of $T^2_n$ that resulted in a large average filling-volume.  Given a tree $T$ in $K^d_n$ and a $d$-simplex $\tau \in T$,  the set of all $d$-simplices in $K^d_n$ such that $T \oplus \tau \oplus \sigma$ is a tree,  is called the {\em cut of $\tau$ in $T$}, denoted by $\Cut(\tau, T)$.  

We constructed $T^2_n$ by repeated conical extensions of an arbitrary $2$-tree $T^2_{\frac{n}{2} + 1}$ on $[\frac{n}{2}+1]$ using the sequence of $1$-trees depicted in Figure~\ref{figTwoDTree}. It should be clear that we only used the properties of this sequence of $1$-trees to demonstrate that the average filling-volume of $T^2_n$ is $\Omega(n^2)$. This will be our approach in higher dimensions too. We start with an arbitrary $d$-tree on $[m]$ where this time $m = \ceil{(1 - 1/d)n}$ and then extend it to a $d$-tree on $[n]$ along a sequence of carefully chosen $(d-1)$-trees $T^{d-1}_i$, $i \in \{m, \ldots, n-1\}$. 

If we look at the initial $\frac{n}{2}$-length segment of each of the $1$-trees in Figure~\ref{figTwoDTree}, we see that for every odd $i$, this initial segment consists of a path $X = (1, \ldots, \frac{n}{2})$ with the the center edge $x = \{\frac{n}{4}, \frac{n}{4}+1\}$ replaced by the $2$-length path $\Cone_i(\boundary x)$. Similarly, for even $i$, the initial segment consists of a path  $Y = (\frac{n}{4}, \ldots, 1, \frac{n}{2}, \ldots, \frac{n}{4}+1)$  with the the center edge $y = \{1, \frac{n}{2}\}$ replaced by the $2$-length path $\Cone_i(\boundary y)$. Let us treat $X$ and $Y$, rightfully, as $1$-trees over $[n/2]$ and note the following three useful properties.  Firstly, $\Cut(x, X)$ and $\Cut(y, Y)$ are very large ($n^2/16$, namely of order of the total number of $1$-simplices) and this helps in having a constant fraction of $2$-faces which are ``good'' for the final $2$-tree. Secondly, observe that $\Fill(y, X)$ and $\Fill(x, Y)$ are large ($n/2$; of order of a $1$-tree on $n$ vertices) and hence contribute a large number of simplices to the filling of a good $2$-simplex at each level. Finally observe that $\Fill(y, X)$ contains $x$ and $\Fill(x, Y)$ contains $y$ so that $\Fill(z, T^1_i)$ contains $i-1$ where $z = y$ when $i$ is odd and $z=x$ otherwise, for every $i \in \{\frac{n}{2}+1, \ldots, n-1\}$. Notice also, that the explicit structure of $T_i^1$ on the vertices in $C$ was not used, and may replaced with an arbitrary maximal acyclic extension. We also note that while the fact that $Z_i$ was simple facilitated an explicit description of $\Fill^1_i$, this is not really an essential part of the proof, once one observes the properties of the edges $x$ and $y$ in the odd and even trees respectively.

For $d \geq 3$, it is difficult to explicitly describe a pair of $(d-1)$-trees with the above properties, but we show that from any $(d-1)$-tree with a large average filling-volume, we can construct a pair of $(d-1)$-trees which have the above properties. Firstly, we show that in any tree with a large filling-volume, there is a simplex which  is part of a large number of fillings among which at least one is large. A carefully chosen isomorphism of such a tree will give its pair. Notice that given a tree $T$ and $\tau \in T$, the number of fillings in which $\tau$ participates is $\size{\Cut(\tau, T)}$.

\begin{lemma}
\label{lemmaLargeAvgCapLargef}
In any $d$-tree $T$ on $n$ vertices with $\mu(T) \geq c {n-1 \choose
  d}$, there exist $\tau \in T$ and $\sigma \in \Cut(\tau, T)$ such
that $\size{\Cut(\tau, T)} \cdot \size{\Fill(\sigma, T)} \geq \frac{c^3}{8} {n \choose d+1}{n-1 \choose d}$.
\end{lemma}
\begin{proof}

For any $\tau \in T$, let $f(\tau,  T) = \sum_{\sigma \in \Cut(\tau, T)}{\size{\Fill(\sigma, T)}}$. Let $\sigma^*$ be a $d$-simplex which produces a largest filling in $T$ among all the simplices in $\Cut(\tau, T)$. Since $\size{\Cut(\tau, T)} \cdot {\size{\Fill(\sigma^*, T)}} \geq f(\tau, T)$, by the remark above, it suffices to show that there exists $\tau \in T$ with $f(\tau, T) \geq  \frac{c^3}{8} {n \choose d+1}{n-1 \choose d}$.

If we sum $f(\tau, T)$ over all $\tau \in T$, we add the sizes of fillings of $T$, with each filling being added once for each $d$-face contained in it. That is $\sum_{\tau \in T} f(\tau, T) = \sum_{\sigma \in K^d_n} {\size{\Fill(\sigma, T)}}^2$, which is the square of the $l_2$-norm of the filling-volumes. Since the maximum size of a filling in $T$ is ${n-1 \choose d}$ (the size of $T$) and the average filling-volume of $T$ is at least $c {n-1 \choose d}$ (by assumption), it is not difficult to see that at least $c/2$ fraction of the fillings have a size at least half the average. Hence $\sum_{\tau \in T} f(\tau, T)$ is at least $\left( \frac{c}{2} {n \choose d+1}\right) \left( \frac{c}{2}{n-1 \choose d} \right)^2 = \frac{c^3}{8} {n \choose d+1} {n-1 \choose d}^2$. By averaging there exists at least one $\tau$ in $T$ which satisfies our required bound.  \end{proof}

By a systematic extension of the $2$-dimensional construction, keeping intact the essential features noted above, we establish the following theorem; a complete proof of which is given in the appendix.

\begin{theorem}
\label{theoremTreesOfLargeAvgCapSize}
For any two positive integers $d$ and $n$, there exists a $d$-tree $T$ on $[n]$ with $\mu(T) \geq c_d {n-1 \choose d}$, where $c_d = 16(48)^{-3^{d-1}}$.
\end{theorem}

\section{A remark on cycles and hypertrees over $\Q$.}
\label{sectionRationalRemarks}

We conclude by pointing out that our results over $\F_2$ extend to $\Q$. Although cycles over $\F_2$ are not necessarily cycles over $\Q$, one can verify that if $Z$ is a simple $\F_2$-cycle, then it is contained in the support of a simple $\Q$-cycle. This follows from the $\F_2$-acyclicity of $Z' = Z \setminus \{\sigma\}$ for any simplex $\sigma \in Z$, which in turn implies $\Q$-acyclicity for $\Z'$. Hence (i)  a Hamiltonian $d$-cycle over $\Q$ exists whenever a Hamiltonian $d$-cycle over $\F_2$ exists and (ii) the average filling-volume of a $d$-tree $T$ computed over $\Q$ is at least as big as the average filling-volume of $T$ computed over $\F_2$.

It is possible that simple cycles over $\Q$ may be larger than largest simple cycles over $\F_2$. We know that this is indeed the case when $d = n-4$ from a work on hypercuts, which are duals of simple cycles \cite{linial2014Shadows}.
We believe that for $d =2$, we can have an $n$-vertex Hamiltonian cycle over $\Q$ for all large  $n$, in contrast with the $\F_2$-case. Though the construction in the proof of Theorem~\ref{theoremLargeForestGivenBoundary} is done over $\F_2$, with a proper generalization of the cone and link operators, we can do this construction over any field. Let us call a chain over $\Q$ whose support is a $\Q$-tree a {\em $\Q$-weighted tree}.  There are no parity restrictions over $\Q$, but one can show that when $n \leq 5$, all the boundaries forbidden for $n$-vertex $2$-trees over $\F_2$ are forbidden for $n$-vertex $\Q$-weighted $2$-trees too. One can establish that these are the only restrictions over $\Q$. That is, 

\begin{claim}
Given any $1$-cycle $Z$ over $\Q$ on $[n]$, $n \geq 6$, we can construct a $\Q$-weighted tree $T$ on $[n]$ with $\boundary T = Z$. In particular, $n$-vertex Hamiltonian $2$-cycles over $\Q$ exist for every $n \geq 6$.
\end{claim}

The proof idea is to use a case analysis to establish the claim for $n=6$ and then use the conical extension to construct a $\Q$-weighted tree whose boundary is $Z$. This will have an immediate effect in higher dimensions on the co-rank term in a result for $\Q$ analogous to Theorem~\ref{theoremLargeForestGivenBoundary}.

\bibliography{../../deepak}

\begin{thebibliography}{10}

\bibitem{aronshtam2015threshold}
Lior Aronshtam and Nathan Linial.
\newblock The threshold for $d$-collapsibility in random complexes.
\newblock {\em Random Structures \& Algorithms}, 2015.

\bibitem{aronshtam2015TopHomology}
Lior Aronshtam and Nathan Linial.
\newblock When does the top homology of a random simplicial complex vanish?
\newblock {\em Random Structures \& Algorithms}, 46(1):26--35, 2015.

\bibitem{aronshtam2013collapsibility}
Lior Aronshtam, Nathan Linial, Tomasz {\L}uczak, and Roy Meshulam.
\newblock Collapsibility and vanishing of top homology in random simplicial
  complexes.
\newblock {\em Discrete \& Computational Geometry}, 49(2):317--334, 2013.

\bibitem{babson2011fundamental}
Eric Babson, Christopher Hoffman, and Matthew Kahle.
\newblock The fundamental group of random 2-complexes.
\newblock {\em Journal of the American Mathematical Society}, 24(1):1--28,
  2011.

\bibitem{bjorner1995topological}
Anders Bj{\"o}rner.
\newblock Topological methods. {H}andbook of combinatorics, vol. 1, 2,
  1819--1872, 1995.

\bibitem{cadek2013extending}
Martin {\v{C}}adek, Marek Kr{\v{c}}{\'a}l, Ji{\v{r}}{\'\i} Matou{\v{s}}ek,
  Luk{\'a}{\v{s}} Vok{\v{r}}{\'\i}nek, and Uli Wagner.
\newblock Extending continuous maps: {P}olynomiality and undecidability.
\newblock In {\em Proceedings of the forty-fifth annual ACM symposium on Theory
  of computing}, pages 595--604. ACM, 2013.

\bibitem{carlsson2009topology}
Gunnar Carlsson.
\newblock Topology and data.
\newblock {\em Bulletin of the American Mathematical Society}, 46(2):255--308,
  2009.

\bibitem{dey2006curve}
Tamal~K Dey.
\newblock {\em Curve and surface reconstruction: {A}lgorithms with mathematical
  analysis}, volume~23.
\newblock Cambridge University Press, 2006.

\bibitem{fox2011overlap}
Jacob Fox, Mikhail Gromov, Vincent Lafforgue, and Assaf Naor.
\newblock Overlap properties of geometric expanders.
\newblock In {\em Proc. 22nd ACM-SIAM Symposium on Discrete Algorithms}, number
  EPFL-CONF-170579, pages 1188--1197. Walter De Gruyter \& Co, 2011.

\bibitem{ghrist2014elementary}
Robert Ghrist.
\newblock Elementary applied topology.
\newblock {\em Book in preperation}, 2014.

\bibitem{gromov2010singularities}
Mikhail Gromov.
\newblock Singularities, expanders and topology of maps. {P}art 2: From
  combinatorics to topology via algebraic isoperimetry.
\newblock {\em Geometric and Functional Analysis}, 20(2):416--526, 2010.

\bibitem{herlihy1999topological}
Maurice Herlihy and Nir Shavit.
\newblock The topological structure of asynchronous computability.
\newblock {\em Journal of the ACM (JACM)}, 46(6):858--923, 1999.

\bibitem{kahn1984topological}
Jeff Kahn, Michael Saks, and Dean Sturtevant.
\newblock A topological approach to evasiveness.
\newblock {\em Combinatorica}, 4(4):297--306, 1984.

\bibitem{kalai1983enumeration}
Gil Kalai.
\newblock Enumeration of {$Q$}-acyclic simplicial complexes.
\newblock {\em Israel Journal of Mathematics}, 45(4):337--351, 1983.

\bibitem{kaufman2014ramanujan}
Tali Kaufman, David Kazhdan, and Alexander Lubotzky.
\newblock Ramanujan complexes and bounded degree topological expanders.
\newblock In {\em Foundations of Computer Science (FOCS), 2014 IEEE 55th Annual
  Symposium on}, pages 484--493. IEEE, 2014.

\bibitem{linial2006homological}
Nathan Linial and Roy Meshulam.
\newblock Homological connectivity of random 2-complexes.
\newblock {\em Combinatorica}, 26(4):475--487, 2006.

\bibitem{linial2010SumComplexes}
Nathan Linial, Roy Meshulam, and Mishael Rosenthal.
\newblock Sum complexes - a new family of hypertrees.
\newblock {\em Discrete \& Computational Geometry}, 44(3):622--636, 2010.

\bibitem{linial2014phase}
Nathan Linial and Yuval Peled.
\newblock On the phase transition in random simplicial complexes.
\newblock {\em arXiv preprint arXiv:1410.1281}, 2014.

\bibitem{linial2014Shadows}
Nati Linial, Ilan Newman, Yuval Peled, and Yuri Rabinovich.
\newblock Extremal problems on shadows and hypercuts in simplicial complexes.
\newblock {\em arXiv preprint arXiv:1408.0602}, 2014.

\bibitem{lubotzky2015random}
Alexander Lubotzky and Roy Meshulam.
\newblock Random {L}atin squares and 2-dimensional expanders.
\newblock {\em Advances in Mathematics}, 272:743--760, 2015.

\bibitem{matousek2014embeddability}
Ji{\v{r}}{\'\i} Matou{\v{s}}ek, Eric Sedgwick, Martin Tancer, and Uli Wagner.
\newblock Embeddability in the $3$-sphere is decidable.
\newblock In {\em Proceedings of the thirtieth annual symposium on
  Computational geometry}, page~78. ACM, 2014.

\bibitem{matousek2011hardness}
Ji{\v{r}}{\'\i} Matou{\v{s}}ek, Martin Tancer, and Uli Wagner.
\newblock Hardness of embedding simplicial complexes in $\mathbb{R}^d$.
\newblock {\em Journal of the European Mathematical Society}, 13(2):259--295,
  2011.

\bibitem{matousek2008borsuk}
Ji\v{r}{\'i} Matou{\v{s}}ek.
\newblock {\em Using the Borsuk-Ulam theorem: {L}ectures on topological methods
  in combinatorics and geometry}.
\newblock Springer Science \& Business Media, 2008.

\bibitem{meshulam2009homological}
Roy Meshulam and Nathan Wallach.
\newblock Homological connectivity of random $k$-dimensional complexes.
\newblock {\em Random Structures \& Algorithms}, 34(3):408--417, 2009.

\bibitem{newman2013multiplicative}
Ilan Newman and Yuri Rabinovich.
\newblock On multiplicative $\lambda$-approximations and some geometric
  applications.
\newblock {\em SIAM Journal on Computing}, 42(3):855--883, 2013.

\bibitem{newman2015connectivity}
Ilan~I Newman and Yuri Rabinovich.
\newblock On connectivity of the facet graphs of simplicial complexes.
\newblock {\em arXiv preprint arXiv:1502.02232}, 2015.

\bibitem{saks2000wait}
Michael Saks and Fotios Zaharoglou.
\newblock Wait-free k-set agreement is impossible: The topology of public
  knowledge.
\newblock {\em SIAM Journal on Computing}, 29(5):1449--1483, 2000.

\end{thebibliography}

\clearpage
\appendix

\section{Proof of Theorem~\ref{theoremTreesOfLargeAvgCapSize}}

\paragraph{Statement.} For any two positive integers $d$ and $n$, there exists a $d$-tree $T$ on $[n]$ with $\mu(T) \geq c_d {n-1 \choose d}$, where $c_d = 16(48)^{-3^{d-1}}$.

\begin{proof} 
For $d=1$, the required tree $T$ is the Hamiltonian path on $n$ vertices. It is easy to verify that $\mu(T) \geq \frac{1}{3}(n-1)$ as claimed. Hence let $d \geq 2$ and we assume that the statement of the theorem is true in dimension $d-1$. In particular, we will use a $(d-1)$-tree $X$ on $m$ vertices, $m = \ceil{(1 - 1/d)n}$, with large average filling-volume in the construction of $T$. Since $c_d {n-1 \choose d} \leq 1$ for $n < 10d$, we assume $n \geq 10d$. We set $c = c_{d-1}$ (for readability). 

We construct a sequence of $d$-trees, $T^d_{m+1}, T^d_{m+2}, \ldots, T^d_n = T$, where $T^d_{m+1}$ is an arbitrary
$d$-tree on $[m+1]$ and each $T^d_{i+1}$, $i \in \{m+1, \ldots, n-1\}$, is a conical extension of $T^d_{i}$. That is,

\begin{equation}
	T^d_{i+1} = \Ext_{i+1}(T^d_{i}, T^{d-1}_{i}), \quad i
        \in \{m+1, \ldots, n-1\}, 
\end{equation}
where $T^{d-1}_{i}$ is a $(d-1)$-tree on $[i]$ which we construct as described next.

For each $i \geq m+1$, we construct $T^{d-1}_i$ based on a $(d-1)$-tree $\Xan$ on the vertex set $[m]$ which has a large average filling-volume, i.e., $\mu(\Xan) \geq c {m - 1 \choose d-1}$. Such a tree $X$ exists by the induction hypothesis. By Lemma~\ref{lemmaLargeAvgCapLargef}, there exist $x \in \Xan$  and $y \in \Cut(x, \Xan)$ such that $\size{\Cut(x, \Xan)} \cdot \size{\Fill(y, \Xan)} \geq \frac{c^3}{8} {m-1 \choose d-1} {m \choose d}$. Let $\Yan$ be a $(d-1)$-tree isomorphic to $\Xan$ obtained from $\Xan$ by relabelling its vertices so that $x$ and $y$ exchange their roles. That is, $y \in \Yan$ and $x \in \Cut(y, \Yan)$ with $\size{\Cut(y, \Yan)} \size{\Fill(x, \Yan)} \geq \frac{c^3}{8} {m-1 \choose d-1} {m \choose d}$.  For an odd number $i$ in $\{m +1, \ldots, n-1 \}$, we construct a forest $F_i$ on the vertex set $[m] \cup \{i\}$ by removing $x$ from $X$ and  adding $\Cone_i (\boundary x)$ in its place, i.e., $F_i = \Xan \oplus x \oplus \Cone_i(\boundary x)$. Note that since $X$ is acyclic, $F_i$ is acyclic. Though it is not a tree  on $[m] \cup \{i\}$ yet, any $(d-1)$-face $\sigma \in \Cut(x, X)$ creates a cycle in $F_i$. More precisely, for any $\sigma \in \Cut(x, X)$, 
\begin{equation} 
\label{eqnCapFiOdd} 
\Fill(\sigma, F_i) = \Fill(\sigma, X) \oplus x \oplus \Cone_i (\boundary x).
\end{equation}

Similarly, for an even number $i$ in $\{m +1, \ldots, n-1 \}$, we construct a forest $F_i$ on the vertex set $[m] \cup \{i\}$ by removing $y$ from $Y$ and  adding $\Cone_i(\boundary y)$ in its place, i.e., $F_i = \Yan \oplus y \oplus \Cone_i(\boundary y)$. Note that $F_i$ is acyclic in this case too and  for any $\sigma \in \Cut(y, Y)$,
\begin{equation}
\label{eqnCapFiEven}
	\Fill(\sigma, F_i) = 
	\Fill(\sigma, Y) \oplus y \oplus \Cone_i(\boundary y), 
\end{equation}

Finally, we let $T^{d-1}_i$ be any $(d-1)$-tree on $[i]$ containing $F_i$. This completes the construction of $T^d_n$ and we are all set to estimate $\mu(T^d_n)$. We note here that this construction is a generalization of that in the explicit proof for $d=2$.  The forests $F_i$ are just the trees $T_i$ with out the extension to the part in $C$ and the faces $x$ and $y$ are just edges $\{\frac{n}{4}, \frac{n}{4}+1\}$ and $\{1,\frac{n}{2} \}$ respectively. 

\begin{insideclaim}
\label{claimGoodCaps}
Let $\Gamma$ be the set of (``good'') $d$-faces of the form $\{i\} \cup \gamma'$ where $i \in \{m + 3, \ldots, n\}$ and $\gamma'$ belongs to $\Cut(x, \Xan)$ if $i$ is even and $\Cut(y, \Yan)$ if $i$ is odd. Then, for each $\gamma = \{i\} \cup \gamma' \in \Gamma$,
\[
	\size{\Fill(\gamma, T^d_n)} \geq \size{\Fill(y, X)} \cdot (i - m -2).
\]
\end{insideclaim}

We prove the above claim by establishing that $\Fill(\gamma, T^d_n)$ contains the following ``large'' set $K_{\gamma}$ defined below.
\begin{equation}
\label{eqnDefineKgamma}
K_{\gamma} = 
	\bigcupdot_{j = m + 1}^{i-1}
		\Cone_{j+1} \left( \Fill(\sigma_j, F_{j}) \right),
\end{equation}
where $\sigma_{i-1} = \gamma'$ and for all $j < i-1$, $\sigma_j$ is
$y$ when $j$ is odd and $x$ otherwise. Notice that $\size{K_{\gamma}}
\geq \size{\Fill(y, X)} \cdot (i - m -2)$ and hence we will establish Claim~\ref{claimGoodCaps} if we show that $K_{\gamma}$ is contained in $\Fill(\gamma, T^d_n)$. It follows from Equations (\ref{eqnCapFiOdd}) and (\ref{eqnCapFiEven}) that for $j \geq m + 1$, 
\begin{eqnarray}
\label{eqnBoundaryPart}
\boundary \Cone_{j+1} \left(\Fill(\sigma_j, F_j) \right)
	&=& \Cone_{j+1} (\boundary\sigma_j) 
			\oplus \Fill(\sigma_j, F_j) 
			\\
	&=& \Cone_{j+1}(\boundary\sigma_j) 
			\oplus H_j
			\oplus \Cone_j (\boundary\sigma_{j-1})
			\nonumber
\end{eqnarray}
where $H_j$ is some set of $(d-1)$-simplices on $[m]$. Substituting Equation~(\ref{eqnBoundaryPart}) in Equation~(\ref{eqnDefineKgamma}) results in a partial telescoping giving rise to the following.
\[
\begin{array}{rcl}
\boundary(K_{\gamma}) 
	&=& \Cone_i (\boundary\gamma' ) \oplus 
		\Cone_{m+1} (\boundary\sigma_{m}) \oplus 
		\sum_{j = m+1}^{i-1}H_j \\
	&=& \boundary\gamma \oplus  
		\gamma' \oplus
		\Cone_{m+1} (\boundary\sigma_{m}) \oplus
		\sum_{j = m+1}^{i-1}H_j \\
	&=& \boundary\gamma \oplus Z
\end{array}
\] 
where $Z$ is a collection of $(d-1)$-simplices on $[m+1]$. Since $Z =
\boundary(K_{\gamma}) + \boundary\gamma$, being a boundary,  is a $d$-cycle on $[m + 1]$ and $T^d_{m+1}$ is a tree on $[m+1]$, there exists a unique filling $K'$ of $Z$ in $T^d_{m+1}$. Then, since $K_{\gamma}$ and $K'$ are disjoint, $\boundary(K_{\gamma} \cup K') = \boundary\gamma$. This means, by definition, that $K_{\gamma} \cup K' = \Fill(\gamma, T^d_n)$ and it completes the proof of Claim~\ref{claimGoodCaps}.

We obtain the required lower bound on the sum of filling-volumes by adding the filling-volumes of all fillings generated by $d$-faces in $\Gamma$. The details of estimation are below.

\begin{equation*}
\begin{array}{rcll}
{n \choose d+1} \mu(T^d_n) 
&\geq&
	\sum_{\gamma \in \Gamma}
	\size{\Fill(\gamma, T^d_n)}\\
&\geq&
 	\sum_{i=m+3}^{n}
	\size{\Cut(x, \Xan)} \cdot
	\size{\Fill(y, \Xan)}{(i - m -2)}
	& (\textnormal{Claim \ref{claimGoodCaps}}) \\
&=&
	\size{\Cut(x, \Xan)}\cdot
	\size{\Fill(y, \Xan)}\cdot
	\frac{1}{2} \left(\frac{n}{d}-3\right) \left(\frac{n}{d}-2\right)
	& \left(n - m \geq \frac{n}{d} - 1\right) \\
&\geq&
	\frac{c^3}{2^4} 
	{m \choose d}
	{m  -1 \choose d-1}
	\frac{(n-3d)(n-2d)}{d^2}
	& \textnormal{(By choice of $x$ and $y$)}\\
&\geq&
	\frac{c^3}{2^4} \left(1 - \frac{1}{d}\right)^{2d-1}
	{n-1 \choose d}
	{n-2 \choose d-1}
	\frac{(n-3d)(n-2d)}{d^2}
	& \textnormal{(Observation \ref{obsBinomialBound})}\\
&\geq&
	\frac{c^3}{2^4} \left(1 - \frac{1}{d}\right)^{2d-1}
	{n-1 \choose d}
	\frac{1}{4}
	{n \choose d+1}
	& (n \geq 10d)\\ 
&\geq&
	\frac{c^3}{2^8} 
	{n-1 \choose d}
	{n \choose d+1} 
	& (\left(1 - \frac{1}{d}\right)^{2d-1} \geq \frac{1}{8}, \forall d \geq 2). 
\end{array}
\end{equation*}

Notice that $\frac{1}{2^8} (c_{d-1})^3 = c_d$.
\end{proof}

\begin{observation}
\label{obsBinomialBound}
Let $d$ and $n$ be any two positive integers. Then for $\alpha = (1-1/d)$ and $m = \ceil{\alpha n}$  
\[
\begin{array}{rcl}
	{m \choose d} 	&\geq& \alpha^d {n-1 \choose d},\; ;\\
	{m-1 \choose d-1} &\geq& \alpha^{d-1} {n-2 \choose d-1}.
\end{array}
\] 
\end{observation}
\begin{proof}
Both inequalities above follow by noting that $(\alpha n - i) = \alpha(n - i/\alpha) \geq \alpha (n - i - 1)$ for all $i \in \{0, \ldots, d-1\}$ and then expanding the binomial coefficient on the left side. 
\end{proof}

\end{document}